\documentclass[pdflatex,sn-mathphys-num]{sn-jnl}

\usepackage{graphicx}%
\usepackage{multirow}%
\usepackage{amsmath,amssymb,amsfonts}%
\usepackage{amsthm}%
\usepackage{mathrsfs}%
\usepackage[title]{appendix}%
\usepackage{xcolor}%
\usepackage{textcomp}%
\usepackage{manyfoot}%
\usepackage{booktabs}%
\usepackage{algorithm}%
\usepackage{algorithmicx}%
\usepackage{algpseudocode}%
\usepackage{listings}%
\usepackage{subcaption} 

\theoremstyle{thmstyleone}%
\newtheorem{lemma}{Lemma}

\theoremstyle{thmstyletwo}%
\newtheorem{remark}{Remark}%

\theoremstyle{thmstylethree}%

\begin{document}

\title[Supercritical Mass Fokker--Planck]{Supercritical Mass and Condensation in Fokker--Planck Equations for Consensus Formation}

\author[1]{\fnm{Monica} \sur{Caloi}}\email{monica.caloi@unimore.it}

\author*[2]{\fnm{Mattia} \sur{Zanella}}\email{mattia.zanella@unipv.it}

\affil[1]{\orgdiv{Department of Physics, Informatics and Mathematics}, \orgname{University of Modena and Reggio Emilia}, \orgaddress{\street{Via Campi 213/b}, \city{Modena}, \postcode{41125},  \country{Italy}}}

\affil[2]{\orgdiv{Department of Mathematics "F. Casorati"}, \orgname{University of Pavia}, \orgaddress{\street{Via A. Ferrata, 5}, \city{Pavia}, \postcode{27100}, \country{Italy}}}


\abstract{Inspired by recently developed Fokker--Planck models for Bose--Einstein statistics, we study a consensus formation model with condensation effects driven by a polynomial diffusion coefficient vanishing at the domain boundaries. For the underlying kinetic model, given by a nonlinear Fokker--Planck equation with superlinear drift, it was shown that if the initial mass exceeds a critical threshold, the solution may exhibit finite-time concentration in certain parameter regimes. Here, we show that this supercritical mass phenomenon persists for a broader class of diffusion functions and provide estimates of the critical mass required to induce finite-time loss of regularity.}


\maketitle

\section{Introduction}\label{intro}

In recent decades, several models have been introduced to describe the emergence of consensus in large interacting particle systems \cite{DG,HK}. Consensus-type phenomena typically result from the interplay between random effects, modelled by diffusion operators, and compromise mechanisms, described through drift-type terms. For large interacting systems, kinetic and mean-field formulations can be derived, enabling the analysis of the time evolution of a density function representing the proportion of agents or particles concentrated around an emerging collective state \cite{APZ,BHW,D_etal,DW,MT,PTTZ,T06}.
These modelling frameworks have found significant applications in the analysis of opinion formation processes characterising multiagent systems and their societal impact, see, e.g., \cite{ACDZ,DFWZ,DWo,Z}. Furthermore, they have played a central role in the development of global optimisation methodologies and in the rigorous investigation of the collective alignment, see, e.g., \cite{CCTT,MR4793478,LTZ,FPZ} and the references therein.

In this work, we focus on a new class of opinion formation models that are capable to describe strong concentration phenomena in the form of condensation, see \cite{CDTZ}. To this end, similarly to the Kaniadakis-Quarati equation for quantum indistinguishable particles relaxing towards the Bose-Einstein statistics \cite{BAGT,CHW,KQ,Hopf,T12}, we study the effects of  superlinear drifts on the emerging distribution. We also mention \cite{CDFT} for the connected drift equation in which the concentrated mass is strictly increasing after the blow-up time. Hence, we extend the argument provided in \cite{TZ} for the blow-up of the solution in finite time to a general polynomial diffusion coefficient. Loss of regularity is triggered for a sufficient initial mass in the supercritical parameters' regime. 

In more detail, the paper is structured as follows: in Section \ref{superlinear} we present the general Fokker--Planck framework, including linear and superlinear drift cases, and discuss the structure of stationary states for different diffusion weights. In Section \ref{loss}, we analyse the propagation of $L^2$-regularity for the time-dependent problem and establish conditions for finite-time loss of regularity in the supercritical mass regime. Finally, we summarise our findings and outline open questions and possible extensions.
 
\section{Superlinear Fokker-Planck Equations}\label{superlinear}
The classical set-up of problems for consensus formation relies on the evolution of a distribution $f = f(w,t)$, $w \in I= [-1,1]$, $t\ge0$,  solution to the Fokker-Planck-type equation 
\begin{equation}
\label{eq:generalFP}
\begin{split}
&\dfrac{\partial}{\partial t}f(w,t) = \dfrac{\partial}{\partial w} \left[ (w-m)J(f) + \sigma^2 \dfrac{\partial}{\partial w}(H(w) f(w,t))\right],  \\
&f(w,0) = f_0(w) \in L^1(I)
\end{split}
\end{equation}
complemented with no-flux boundary conditions and with $\sigma^2>0$. In \eqref{eq:generalFP} the weight $H(w)$ is such that $H(\pm 1) = 0$. As in \cite{T06}   a classic choice for this quantity is provided by 
\begin{equation}
\label{eq:HBeta}
H(w) = 1-w^2.  
\end{equation}
In the linear case, corresponding to $J(f) = f$, it is easily verified that the equilibrium distribution corresponds, under the assumption \eqref{eq:HBeta}, to a Beta-type distribution 
\begin{equation}
\label{eq:beta_equi}
f^\infty(w) = \dfrac{\mu}{2^{a+b+1}\textrm{B}(a+1,b+1)} (1+w)^{a}(1-w)^{b},
\end{equation}
where $\mu>0$ is the initial conserved mass and
\begin{equation}\label{eq:ab}
a = -1 + \dfrac{1+m}{2\sigma^2}, \qquad b = -1 +\dfrac{1-m}{2\sigma^2}. 
\end{equation}
being $\textrm{B}(\cdot,\cdot)$ the Beta function.\\

As discussed in \cite{T06}, the solution to \eqref{eq:generalFP} conserves mass and momentum. For the problem of the convergence of the solution to the linear consensus model towards the equilibrium density \eqref{eq:beta_equi} we mention \cite{FPTT}. For data-oriented applications, the propagation of $L^2$ regularity has important consequences in connection with uncertainty quantification \cite{BCMZ,CPZ,DPZ}. In the present setting, the following result holds
\begin{lemma}
Let $f(w,t)$, $w \in I$, $t\ge0$, be the solution to the linear Fokker-Planck equation for consensus formation \eqref{eq:generalFP} with $J(f) = f$  and nonconstant diffusion weight \eqref{eq:HBeta} complemented with no-flux boundary conditions. We assume that the initial distribution satisfies $f_0(w) \in L^1(I)\cap L^2(I)$. Then, for all $t \in [0,T]$, we have $ f(\cdot,t)\in L^1(I)\cap L^2(I)$.  Moreover, if $m= 0$ and $\sigma^2\le\frac{1}{2}$, for any initial distribution $f_0 \in L^1(I)\cap L^2(I) $ such that  $\int_{I} w f_0(w)dw = 0$, the solution satisfies the uniform-in-time estimate
\[
\| f(\cdot,t)\|_{L^2(I)}^2 \le \max \left\{\| f_0\|_{L^2(I)}^2; \left(\dfrac{C_N\mu^4(1-2\sigma^2)}{2\sigma^2}\right)^{1/2} \right\},\qquad t\ge0,
\] 
where $C_N = \frac{3^3}{2^5}$. 
\end{lemma}
\begin{proof}
Let $f_0(w) = f(w,0) \in L^1(I) \cap L^2(I)$, $f_0(w)\ge0$, be the initial distribution for the linear consensus formation model \eqref{eq:generalFP} with initial mass $\int_{I}f(w,0)dw = \mu>0$. For any $\sigma^2>0$ and $m \in (-1,1)$ we have
\[
\dfrac{\partial}{\partial t}\| f(\cdot,t)\|_{L^2(I)}^2 = \int_{I}2f\dfrac{\partial}{\partial w}\left[(w-m)f + \sigma^2 \dfrac{\partial}{\partial w}((1-w^2)f) \right] dw, 
\]
from which we get
\begin{equation}
\label{eq:evoL2_lin}
\dfrac{\partial}{\partial t}\| f(\cdot,t)\|_{L^2(I)}^2 = (1-2\sigma^2)\| f\|_{L^2(I)}^2 - 2\sigma^2 \int_{I}(1-w^2)\left(\dfrac{\partial}{\partial w}f\right)^2dw. 
\end{equation}
Hence, for any $m\in (-1,1)$, from the Gronwall's Lemma we get
\[
\| f(\cdot,t)\|^2_{L^2(I)} \le \| f_0\|_{L^2(I)}^2 e^{(1-2\sigma^2)t},
\]
and $f \in L^2(I)$ for any $t \in[0,T]$. 
Finally, we observe that if $\int_{I}wf(w,0)dw =0$ and $m=0$ the solution to \eqref{eq:generalFP} is such that $\int_{I}wf(w,t)dw =0$ for all $t>0$. Hence,
Lemma A1 in \cite{TZ} provides the following Nash-type inequality in the presence of the nonconstant diffusion weight \eqref{eq:HBeta}
\[
\left(\int_{I}|f(w,t)|^2dw\right)^3 \le C_N \int_{I}(1-w^2)\left(\dfrac{\partial}{\partial w}f(w,t) \right)^2dw \left(\int_{I}|f(w,t)|dw\right)^4,
\]
with $C_N = \frac{3^3}{2^5}$. By nonnegativity of $f$ and mass conservation, $\int_I|f(w,t)|\,dw=\mu$ for any $t>0$. Therefore, from \eqref{eq:evoL2_lin} we obtain
\[
\dfrac{\partial}{\partial t}\| f(\cdot,t)\|_{L^2(I)}^2 \le  \| f(\cdot,t)\|_{L^2(I)}^2\left[(1-2\sigma^2) - \dfrac{2\sigma^2}{\mu^4C_N}\| f(\cdot,t)\|_{L^2(I)}^4\right], 
\]
from which we conclude. 
\end{proof}

As a result, the solution to the linear consensus model belongs to $L^2(I)$ for any initial, conserved, mass $\mu>0$.  Recently, consensus formation dynamics have been studied under the lens of mass-related clustering effect. To this end, inspired by the Kaniadakis-Quarati equation \cite{KQ,T12}, in the recent work  in \cite{CDTZ} it has been considered the superlinear case $J(f) = f(1+\beta H^\alpha(w)f^\alpha)$, $\beta>0, \alpha>0$. Under this assumption, the resulting consensus model \eqref{eq:generalFP} is still mass preserving and admits an equilibrium distribution, solution to the following differential equation
\begin{equation}
\label{eq:equation_finfty}
(w-m)f^\infty(w)(1+\beta (H(w)f^\infty(w)))^\alpha+ \sigma^2 \dfrac{\partial}{\partial w}(H(w)f^\infty(w)) = 0
\end{equation}
which is provided, for any $C>0$, by 
\[
f^\infty_C(w) =C \dfrac{(1+w)^{a}(1-w)^b}{(1-\beta \left(C(1+w)^{ a}(1-w)^{ b}\right)^\alpha)^{1/\alpha}},\qquad w\in I, 
\] 
with $a,b$ as in \eqref{eq:ab} and
\[
C \le \bar C = \left(\dfrac{1}{\beta}\right)^{1/\alpha}(1+m)^{-a-1}(1-m)^{-b-1}. 
\]
It is interesting to observe that, in correspondence to the value $\bar C$, the steady state $f^\infty_{\bar C}(w)$ blows up around $w=m$ with order $2/\alpha$ as $w \to m$. Indeed, we have 
\[
\lim_{w \to m}\dfrac{1-\beta \left(\bar C(1+w)^{ a}(1-w)^{ b}\right)^\alpha}{(w-m)^2} = \dfrac{\alpha}{2\sigma^2(1-m^2)}. 
\]
As a result, a finite critical mass exists and is given by 
\[
\mu = \int_{I} f^\infty_{\bar C}(w) dw<+\infty, 
\]
for any $\alpha>2$. Hence, at variance with the linear case the initial mass of the system plays an important role in the propagation of regularity for this class of problems. \\

We now extend the previous analysis to diffusion weights of the form
\[
H(w) = (1-w^2)^{\gamma}, \qquad \gamma>1,
\]
in order to assess the impact of diffusion localization on the structure of stationary states
and on the existence of a critical mass.
As shown in Figure~\ref{fig:H}, increasing $\gamma$ yields a progressively more localized diffusion profile.

\begin{figure}
	\centering	\includegraphics[width=0.7\textwidth]{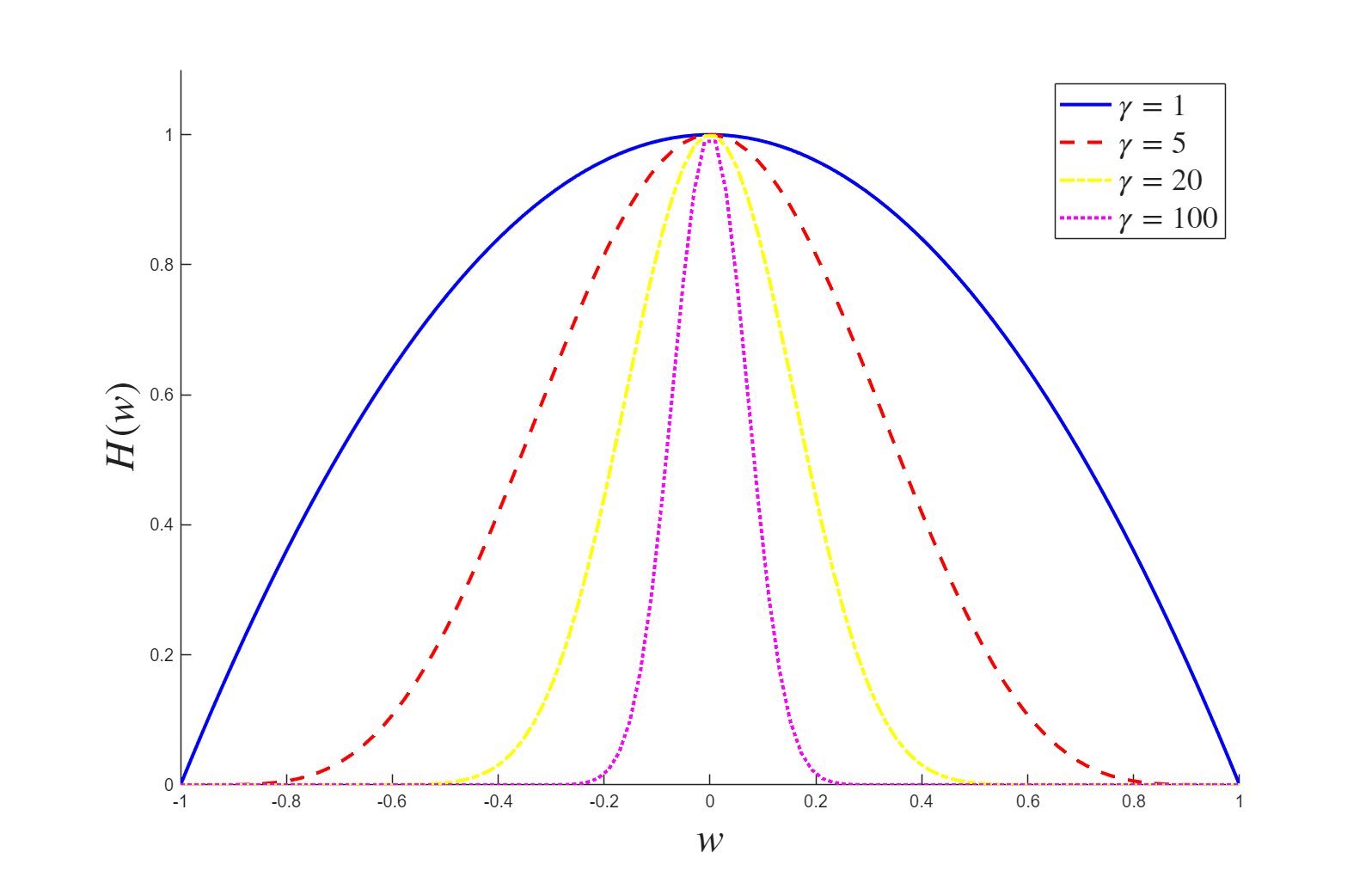}
	\caption{Shape of $H(w) = (1-w^2)^{\gamma}$ for different values of $\gamma$.}
	\label{fig:H}
\end{figure}

In this setting, the stationary solution of the superlinear consensus model \eqref{eq:generalFP}
can be formally computed from \eqref{eq:equation_finfty} and is given by 
\[
f^{\infty}_{C}(w)
=
\frac{1}{H(w)}
\frac{Ce^{-\frac{1}{\sigma^2}V(w)}}
{\left(1-\beta C^{\alpha}e^{-\frac{\alpha}{\sigma^2}V(w)}\right)^{1/\alpha}},
\]
where $V(\cdot) $ is such that 
$$
\dfrac{\partial}{\partial w}V(w) = \dfrac{w-m}{(1-w^2)^{\gamma}}, 
$$
i.e.
\[
V(w) = \frac{1}{2(\gamma-1)} \frac{1}{(1-w^2)^{\gamma-1}} -m  \mathcal{I}_\gamma(w),
\]
where
\[
\mathcal{I}_\gamma(w) = \int \frac{1}{(1-w^2)^\gamma} \, dw.
\]
For integer values of $\gamma\in \mathbb N,\gamma \ge 1$, the function $\mathcal{I}_\gamma(w)$ can be computed recursively as
\begin{align*}
&\mathcal{I}_1(w) = -\frac{1}{2}\log\left(\frac{1-w}{1+w}\right),\\
&\mathcal{I}_{\gamma+1}(w) = \frac{1}{2\gamma}\left((2\gamma-1) \mathcal{I}_\gamma(w) + \frac{w}{(1-w^2)^\gamma}\right), \quad \gamma \in \mathbb{N}, \gamma\geq1.
\end{align*}
In the symmetric case $m=0$ the explicit steady state simplifies to
\begin{equation}
\label{eq:asymptotic_m0}
f^{\infty}_{C}(w)
=
\frac{1}{(1-w^2)^{\gamma}}
\frac{C\, \exp\left\{-\frac{1}{2\sigma^2(\gamma-1)}\frac{1}{(1-w^2)^{\gamma-1}}\right\}}
{\left(1-\beta C^{\alpha}
\exp\left\{-\frac{\alpha}{2\sigma^2(\gamma-1)}\frac{1}{(1-w^2)^{\gamma-1}}\right\}\right)^{1/\alpha}},
\qquad w\in I,
\end{equation}
with
\[
C \leq \bar C
=
\beta^{-1/\alpha}
\exp\!\left(\frac{1}{2\sigma^2(\gamma-1)}\right).
\]
In analogy with the case $\gamma=1$, for \(C = \bar C\), the steady state \(f^{\infty}_{\bar C}\) exhibits a singular behavior
at \(w=0\), blowing up with order \(2/\alpha\):
\[
\lim_{w \rightarrow 0}
\frac{1-\beta \bar{C}^{\alpha}
\exp\!\left(-\frac{\alpha}{2\sigma^2(\gamma-1)}
\frac{1}{(1-w^2)^{\gamma-1}}\right)}{w^2}
=
\frac{\alpha}{2\sigma^2(\gamma-1)}.
\]
Consequently, a finite critical mass
\[
\mu = \int_{I} f^{\infty}_{\bar C}(w)\,dw < +\infty
\]
exists for any \(\alpha>2\).
Finally, for \(m\neq 0\), using the recursive primitive \(\mathcal{I}_\gamma\), the steady state displays a singularity around $w=m$ with analogous qualitative behaviour. 

In Figure~\ref{Fig:f_inf} we illustrate how diffusion localisation and mass criticality affect the stationary solutions in the subcritical regime $C<\bar C$ (left) and  in the critical regime $C=\bar C$ (right). 
We considered the symmetric case $m = 0$ and the stationary state exhibits concentration around $w=0$. We considered $H(w) = (1-w^2)^\gamma$, $w \in [-1,1]$, and several integer values of the constant $\gamma \in \{1,10,50,100\}$. In all cases we fixed $\alpha = 3$, $\beta = 1$ and a diffusion coefficient $\sigma^2 = 0.025$.

\begin{figure}[!htbp]
	\centering
	\begin{subfigure}[b]{0.48\textwidth}
		\centering
		\includegraphics[width=\textwidth]{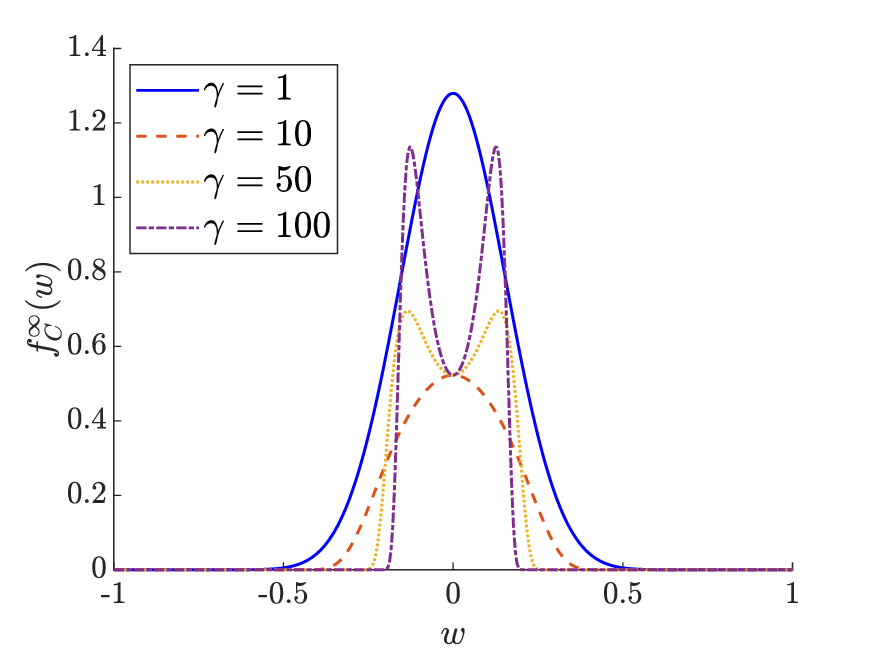}
		\label{fig:finf}
	\end{subfigure}
	\hfill
	\begin{subfigure}[b]{0.48\textwidth}
		\centering
		\includegraphics[width=\textwidth]{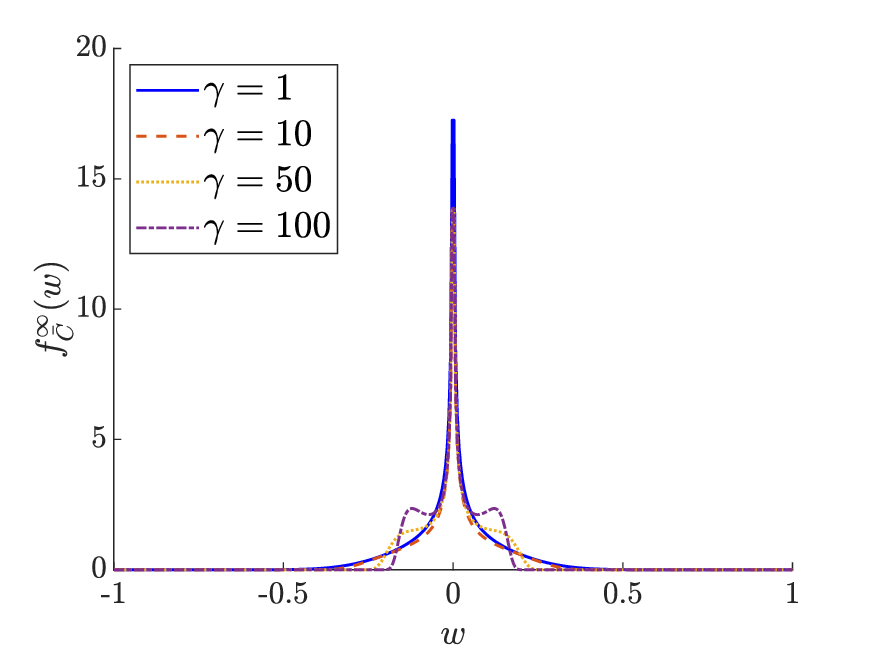}
		\label{fig:finf_critica}
	\end{subfigure}
	\caption{Asymptotic distribution of \eqref{eq:generalFP} with drift $J(f) = f(1+\beta H^\alpha(w)f^\alpha)$ as in \eqref{eq:asymptotic_m0}. We fixed $\alpha=3$, $\beta=1$, $\sigma^2=0.025$, $m=0$ and different values of $\gamma\in \{1,10,50,100\}$. Left: subcritical case $C<\bar{C}$. Right: critical case $C=\bar{C}$.}
	\label{Fig:f_inf}
\end{figure}

In the following we will focus on superlinear Fokker-Planck equations to understand the interplay between the initial mass and the propagation of regularity for this class of problems. 
\section{Loss of regularity and supercritical mass}\label{loss}
We remark that, in the subcritical regime $\alpha<2$, the superlinear Fokker-Planck equation \eqref{eq:generalFP} with $J(f) = f(1+\beta H^\alpha f^\alpha)$, $m=0$ and $\gamma = 1$, is such that the norm $\| f\|_{L^2(I)}^2$ is bounded for all $t>0$ as shown in \cite{TZ} under suitable assumptions. On the other hand in the supercritical regime $\alpha>2$, we get that if the initial mass is sufficiently large
 the solution to the problem loses  $L^2$-regularity in finite time. In this section, we extend the previous result to the case $\gamma>1$. Following the strategy introduced in \cite{TZ}, we show that for $\alpha>2$
solutions to the superlinear Fokker--Planck equation \eqref{eq:generalFP} with
$J(f)=f(1+\beta H^\alpha f^\alpha)$ and $H(w)=(1-w^2)^\gamma$, $\gamma\ge1$ cannot have uniformly bounded $L^2$-norm whenever the initial mass exceeds a suitable critical threshold. 

Specifically, we consider solutions $f$ to
\begin{equation}\label{eq:superlinearFP}
\partial_t f(w,t) =
\partial_w \Bigl[w f(w,t)\bigl(1+\beta (H(w)f(w,t))^\alpha\bigr)
+ \sigma^2 \partial_w\bigl(H(w)f(w,t)\bigr)\Bigr],
\end{equation}
for $w\in I$, with initial datum $f_0\in L^1(I)\cap L^2(I)$. The total mass is conserved,
\[
\mu=\int_I f_0(w)\,dw=\int_I f(w,t)\,dw \qquad \text{for all } t>0.
\]
We first recall a useful estimate. Let $\phi \in L^2(I)$ be a nonnegative function and let $0 < R < 1$. Then
\begin{equation*}
\begin{split}
  \int_{I} \phi(w)\,dw
  &= \int_{|w|\leq R} \phi(w)\,dw + \int_{|w|>R} \phi(w)\,dw \\
  &\leq (2R)^{1/2}\left( \int_{I}\phi^2(w)\,dw\right)^{1/2}
  + \frac{1}{R^2}\int_{I} w^2 \phi(w)\,dw.
\end{split}
\end{equation*}
Optimizing the right-hand side with respect to $R>0$ yields
\begin{equation}\label{eq:L1_L2_energy}
\int_I \phi(w)\,dw
\le 5 \left(\frac{1}{2\sqrt{2}}\right)^{4/5}
\left(\int_I \phi^2(w)\,dw\right)^{\frac{2}{5}}
\left(\int_I w^2\phi(w)\,dw\right)^{\frac{1}{5}}.
\end{equation}
Estimate~\eqref{eq:L1_L2_energy} yields a bound on the $L^1$-norm of a nonnegative function
$\phi \in L^2(I)$ in terms of its $L^2$-norm and its second-order moment.

Based on the following observation, we define the energy in $I$ as the normalised second-order moment
\[
E(t)=\frac{1}{\mu}\int_{I} w^2 f(w,t)\,dw.
\]
Applying \eqref{eq:L1_L2_energy} to $f$, we obtain the lower bound
\begin{equation}\label{eq:energy_L2_ineq}
E(t)\geq \left(\frac{1}{5}\right)^5
\frac{(2\sqrt{2})^4\mu^4}{\|f\|_{L^2(I)}^2}.
\end{equation}
In particular, if $\|f\|_{L^2(I)}$ remains uniformly bounded then $E(t)$ is bounded away from zero.
This relation will be the key ingredient to detect loss of $L^2$-regularity.\\
We write equation \eqref{eq:superlinearFP} in weak form:
\begin{equation*}
       \frac{d}{dt}\int_{I} \varphi(w) f(w,t)\,dw
       = -\int_{I}\varphi'(w)wf\bigl(1+\beta H^{\alpha}f^{\alpha}\bigr)\,dw
       +\sigma^2 \int_{I} \varphi''(w)\,Hf\,dw .
\end{equation*}
By choosing $\varphi(w)=w^2$, we obtain
\begin{equation}\label{eq:weakSLFP}
\begin{split}
  \frac{d}{dt} E(t)
  &=-\frac{2}{\mu}\int_{I} w^2 f\,dw
    -\frac{2\beta}{\mu}\int_Iw^2 H^{\alpha} f^{\alpha+1}\,dw
    +\frac{2\sigma^2}{\mu} \int_{I} Hf\,dw  \\
  &= -2 E(t)
     - \frac{2\beta}{\mu}\int_{I}w^2 H^{\alpha} f^{\alpha+1}\,dw 
     + \frac{2\sigma^2}{\mu} \int_{I} Hf\,dw .
\end{split}
\end{equation}
To estimate the diffusion term, observe that

\begin{equation}\label{eq:ineq_diffusion}
\begin{split}
  \frac{1}{\mu}\int_{I} Hf\,dw
  &\leq \frac{1}{\mu}\int_{I} (1-w^2) f\,dw = 1-E(t)
\end{split}
\end{equation}
Inserting \eqref{eq:ineq_diffusion} into \eqref{eq:weakSLFP}, we obtain
\begin{equation}\label{eq:evol_energy}
 \frac{d}{dt} E(t) \leq
 2\sigma^2-2(1+\sigma^2)E(t)-\frac{2\beta}{\mu} \int_{I}w^2H^{\alpha}f^{\alpha+1}\,dw.
\end{equation}
We estimate the nonlinear term in \eqref{eq:evol_energy} in terms of the energy through an extension of Lemma 3.1 in \cite{TZ}.

\begin{lemma}
Let $\alpha>2$, $\gamma>1$, then the following bound holds
\[
 \int_{I}w^2H^{\alpha}f^{\alpha+1}\,dw
 \geq
 \mu^{\alpha+1}
 \frac{\left(1-E(t)\right)^{\gamma\,\frac{3\alpha}{2}}}
 {(c_{\alpha}d_{\alpha}+d_{\alpha}^{-2})^{\frac{3\alpha}{2}}
  E(t)^{\frac{\alpha-2}{2}}},
\]
where
\[
c_{\alpha}=\left(\frac{2\alpha}{\alpha-2}\right)^{\frac{\alpha}{\alpha+1}},
\qquad
d_{\alpha}=\left(\frac{2(\alpha+1)}{c_{\alpha}(\alpha-2)}\right)^{1/3}.
\]
\end{lemma}
\begin{proof}
Fix $R>0$. We split the integral as
\begin{equation}\label{eq:split_integral}
\begin{split}
 \int_{I} H^{\frac{\alpha}{\alpha+1}} f\,dw
 &= \int_{|w|<R } H^{\frac{\alpha}{\alpha+1}} f\,dw
   + \int_{|w|\ge R} H^{\frac{\alpha}{\alpha+1}} f\,dw \\
 &\le \int_{|w|<R} H^{\frac{\alpha}{\alpha+1}} f\,dw
   + \frac{1}{R^2}\int_{I}w^2 H^{\frac{\alpha}{\alpha+1}} f\,dw .
\end{split}
\end{equation}
For the first term write
\[
H^{\frac{\alpha}{\alpha+1}} f
= H^{\frac{\alpha}{\alpha+1}} |w|^{-\frac{2}{\alpha+1}}
  \bigl(|w|^{\frac{2}{\alpha+1}} f\bigr)
\]
and apply H\"older's inequality with exponents
$p=\alpha+1$ and $q=\frac{\alpha+1}{\alpha}$. This yields
\begin{align*}
 &\int_{|w|<R} H^{\frac{\alpha}{\alpha+1}} f\,dw\le\\
 &\left(\int_{|w|<R} |w|^{-\frac{2}{\alpha}}\,dw\right)^{\frac{\alpha}{\alpha+1}}
 \left(\int_{I}w^2 H^{\alpha} f^{\alpha+1}\,dw\right)^{\frac{1}{\alpha+1}} .
\end{align*}
Since $\alpha>2$, the first integral in the right-hand side is finite and satisfies
\[
\int_{|w|<R} |w|^{-\frac{2}{\alpha}}\,dw
= \frac{2\alpha}{\alpha-2} R^{\frac{\alpha-2}{\alpha}} .
\]
Therefore,
\[
\int_{|w|<R } H^{\frac{\alpha}{\alpha+1}} f\,dw
\le
c_{\alpha}\,
R^{\frac{\alpha-2}{\alpha+1}}
\left(\int_{I}w^2 H^{\alpha} f^{\alpha+1}\,dw\right)^{\frac{1}{\alpha+1}},
\]
where
\[
c_{\alpha}=\left(\frac{2\alpha}{\alpha-2}\right)^{\frac{\alpha}{\alpha+1}}.
\]
Combining this estimate with \eqref{eq:split_integral}, we obtain
\begin{equation}\label{eq:ineq_R_generic}
 \int_{I} H^{\frac{\alpha}{\alpha+1}} f\,dw
 \le
 c_{\alpha} R^{\frac{\alpha-2}{\alpha+1}}
 \left(\int_{I}w^2 H^{\alpha} f^{\alpha+1}\,dw\right)^{\frac{1}{\alpha+1}} +
 \frac{1}{R^2}
 \int_{I}w^2 H^{\frac{\alpha}{\alpha+1}} f\,dw .
\end{equation}
Optimizing the right-hand side with respect to $R>0$ yields
\[
R^*
= d_{\alpha}\,
\frac{\left(\int_{I}w^2 H^{\frac{\alpha}{\alpha+1}} f\,dw\right)^{\frac{\alpha+1}{3\alpha}}}
{\left(\int_{I}w^2 H^{\alpha} f^{\alpha+1}\,dw\right)^{\frac{1}{3\alpha}}},
\qquad
d_{\alpha}=\left(\frac{2(\alpha+1)}{c_{\alpha}(\alpha-2)}\right)^{1/3}.
\]
Substituting $R^*$ into \eqref{eq:ineq_R_generic} we conclude that
\begin{equation}\label{eq:ineq_target}
\begin{split}
 \int_{I} H^{\frac{\alpha}{\alpha+1}} f\,dw
 \le\;&
 (c_{\alpha} d_{\alpha}+d_{\alpha}^{-2}) \\
 &\times
 \left(\int_{I}w^2 H^{\alpha} f^{\alpha+1}\,dw\right)^{\frac{2}{3\alpha}}
 \left(\int_{I}w^2 H^{\frac{\alpha}{\alpha+1}} f\,dw\right)^{\frac{\alpha-2}{3\alpha}} .
\end{split}
\end{equation}
Now, observe that
\begin{equation*}
\int_{I}(1-w^2)f\,dw=\mu -  \int_I w^2 f\,dw = \mu(1-E(t))
\end{equation*}
Then,  we obtain
\begin{equation}\label{eq:Holder}
\begin{split}
\mu^{\gamma}(1-E(t))^{\gamma} 
&= \left( \int_{I} (1-w^2) f(w,t)\,dw \right)^{\gamma} = \left( \int_{I} (1-w^2)\, f^{1/\gamma} f^{(\gamma-1)/\gamma}\,dw \right)^{\gamma} \\
&\le \mu^{\gamma-1}  \int_{I} (1-w^2)^{\gamma} f(w,t)\,dw \le\mu^{\gamma-1} \int_{I} H^{\frac{\alpha}{\alpha+1}} f(w,t)\,dw,
\end{split}
\end{equation}
where the first inequality follows from H\"older's inequality with conjugate exponents $p=\gamma$ and $q=\frac{\gamma}{\gamma-1}$.
Moreover,
\[
\int_{I}w^2 H^{\frac{\alpha}{\alpha+1}} f\,dw
\le \int_{I}w^2 f\,dw
= \mu E(t).
\]
Combining the above estimates with \eqref{eq:ineq_target}, we conclude that
\begin{equation*}
 \int_{I}w^2 H^{\alpha} f^{\alpha+1}\,dw\geq \mu^{\alpha+1}
\frac{\left(1-E(t)\right)^{\gamma\frac{3\alpha}{2}}}
{(c_{\alpha}d_{\alpha}+d_{\alpha}^{-2})^{\frac{3\alpha}{2}} E(t)^{\frac{\alpha-2}{2}}},   
\end{equation*}
which completes the proof.

\end{proof}
This Lemma allows us to show that the evolution of the energy, as given by \eqref{eq:evol_energy}, can reach zero in finite time if the initial mass is sufficiently large. Indeed, incorporating the estimate from Lemma 2 into \eqref{eq:evol_energy} gives
\begin{equation*}
 \frac{d}{dt}E(t) \le 2\sigma^2 - 2(1+\sigma^2)E(t) - \frac{2\beta}{(c_{\alpha}d_{\alpha}+d_{\alpha}^{-2})^{\frac{3\alpha}{2}}} \, \mu^{\alpha} \, 
 \frac{\left(1-E(t)\right)^{\gamma \frac{3\alpha}{2}}}{E(t)^{\frac{\alpha-2}{2}}}.
\end{equation*}
Since $\sigma^2>0$ and $E(t)>0$ for all $t>0$, we can further simplify to the stronger inequality
\begin{equation}\label{eq:evol_energy_new}
\frac{d}{dt}E(t) \le 2\sigma^2 - \frac{2\beta \,\mu^{\alpha}\,\left(1 - E(t) \right)^{\gamma \frac{3\alpha}{2}}}{(c_{\alpha}d_{\alpha}+d_{\alpha}^{-2})^{\frac{3\alpha}{2}} E(t)^{\frac{\alpha-2}{2}}}.
\end{equation}
Given $E(0)$, define
\[
\Psi(\mu) =2\sigma^2 - \frac{2\beta \,\mu^{\alpha}\,\left(1- E(0)\right)^{\gamma \frac{3\alpha}{2}}}{(c_{\alpha}d_{\alpha}+d_{\alpha}^{-2})^{\frac{3\alpha}{2}} E(0)^{\frac{\alpha-2}{2}}}.
\]
Then $\Psi(\mu)$ is negative provided the initial mass satisfies
\[
\mu > \left( \frac{\sigma^2 }{\beta (c_{\alpha}d_{\alpha}+d_{\alpha}^{-2})^{-\frac{3\alpha}{2}}} 
\frac{E(0)^{\frac{\alpha-2}{2}}}{\left(1 - E(0) \right)^{\frac{3\alpha}{2}\gamma}} \right)^{\frac{1}{\alpha}}.
\]
Therefore, as in \cite{TZ}, if $\Psi(\mu)<0$, corresponding to a sufficiently large initial mass, from \eqref{eq:evol_energy_new} we have
\begin{equation*}
 \frac{d}{dt}E(t) \le 2\sigma^2 - \frac{2\beta \,\mu^{\alpha}\,\left(1 - E(0) \right)^{\gamma \frac{3\alpha}{2}}}{(c_{\alpha}d_{\alpha}+d_{\alpha}^{-2})^{\frac{3\alpha}{2}} E(t)^{\frac{\alpha-2}{2}}}
\end{equation*}
This can be rewritten as
\begin{equation}\label{eq:evol_energy_3}
\frac{d}{dt}E(t) \le 2\sigma^2 - \frac{\Lambda}{E(t)^{\eta}}
\le - \frac{\Lambda - 2\sigma^2 E(0)^{\eta}}{E(t)^{\eta}},
\end{equation}
with
\[
\Lambda = \frac{2\beta \, \mu^{\alpha} \left(1-E(0)\right)^{\gamma \frac{3\alpha}{2}}}{(c_{\alpha}d_{\alpha}+d_{\alpha}^{-2})^{\frac{3\alpha}{2}}}, 
\quad
\eta = \frac{\alpha-2}{2} > 0.
\]
From \eqref{eq:evol_energy_3} we deduce
\[
E(t)^{\eta+1} \le E(0)^{\eta+1} - (\eta+1)\left(\Lambda - 2\sigma^2 E(0)^{\eta}\right) t,
\]
so that at time
\[
\bar{t} = \frac{E(0)^{\eta+1}}{(\eta+1)\left(\Lambda - 2\sigma^2 E(0)^{\eta}\right)}
\]
we have $E(\bar{t})=0$. This clearly contradicts inequality \eqref{eq:energy_L2_ineq}, unless the $L^2$-norm of the solution blows up. Notice that $\bar{t}>0$ if the initial energy is sufficiently small and it satisfies
\[
E(0) < \left( \frac{\Lambda}{2\sigma^2} \right)^{1/\eta}.
\]

The condition $E(t)\to 0$ as $t\to\bar t$ implies concentration of the solution at $w=0$. Indeed, by mass conservation and the Cauchy--Schwarz inequality,
\[
\left|\int_I w f(w,t)\,dw\right| \le \mu \sqrt{E(t)},
\]
Consequently, as $t \to \bar{t}$, we get
\[
\int_I \left(w - \int_I w f(w,t)\,dw\right)^2 f(w,t)\,dw \to 0,
\]
and the solution at $t=\bar{t}$ coincides with a Dirac delta concentrated at $w=0$.

\begin{remark}
The above argument can be partially extended to the case $0<\gamma<1$. 
In this regime, a suitable modification of the estimates still yields the existence of a critical mass and finite-time blow-up under analogous assumptions.

The main difference is that estimate \eqref{eq:Holder} is no longer available but can be replaced by the simpler bound
\[
\int_I H^{\frac{\alpha}{\alpha+1}} f\,dw \le \mu(1-E(t)).
\]
As a consequence, the threshold becomes independent of $\gamma$. Recovering its dependence on $\gamma$ would require a more refined analysis.
\end{remark}

Finally, we analyze how the blow-up behavior of the solution to equation \eqref{eq:superlinearFP} depends on the parameter $\gamma\geq1$. The blow-up condition in terms of the initial mass is given by
\[
\mu_\gamma>
\left(
\frac{\sigma^2}{\beta(c_{\alpha}d_{\alpha}+d_{\alpha}^{-2})^{-3\alpha/2}}
\cdot
\frac{E(0)^{(\alpha-2)/2}}{(1-E(0))^{3\alpha\gamma/2}}
\right)^{1/\alpha}.
\]
Since
\[
0<1-E(0)<1,
\]
the denominator decreases exponentially as $\gamma$ increases, and therefore
\[
\lim_{\gamma\to+\infty}\mu_\gamma=+\infty.
\]
This shows that the mass threshold for blow-up is an increasing function of $\gamma$. Consequently, larger values of $\gamma$ require a larger initial mass for the system to develop condensation.
Figure \ref{fig:critical_mass_all} illustrates how  the minimal initial mass required for blow-up varies with the model parameters. The threshold increases with $\gamma$, as illustrated in Figure \ref{fig:mass_gamma_sigma} and \ref{fig:mass_gamma_alpha}, and with the initial energy $E(0)$, as seen in Figure \ref{fig:mass_energy}. In contrast, it decreases as the superlinearity exponent $\alpha$ grows, so that a higher degree of superlinearity requires less mass to induce condensation, as shown by Figures \ref{fig:mass_alpha} and \ref{fig:mass_gamma_alpha}.
\begin{figure}[!htbp]
	\centering
	\begin{subfigure}[b]{0.48\textwidth}
		\centering
		\includegraphics[width=\textwidth]{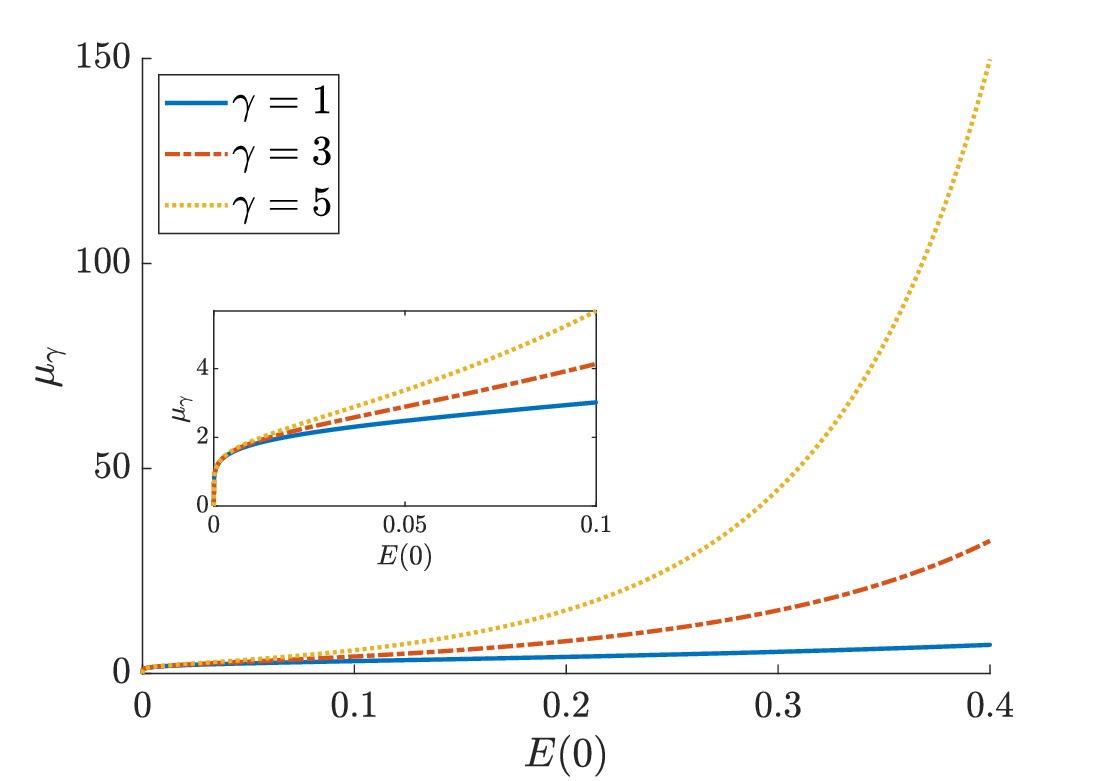}
		\caption{Minimal initial mass as a function of $E(0)$ for different values of $\gamma$ and with fixed $\sigma^2=0.025$, $\beta=1$ and $\alpha=3$.}
		\label{fig:mass_energy}
	\end{subfigure}
	\hfill
	\begin{subfigure}[b]{0.48\textwidth}
		\centering
		\includegraphics[width=\textwidth]{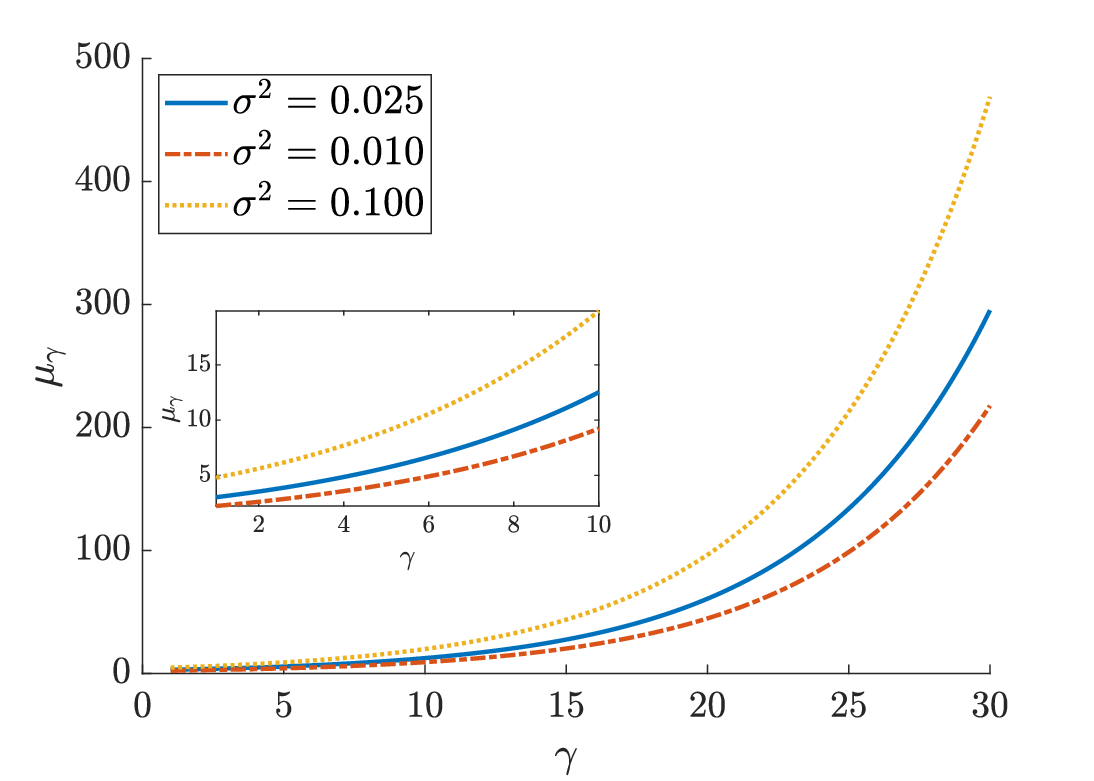}
		\caption{Minimal initial mass as a function of $\gamma$ for different values of $\sigma^2$ and fixed $\alpha=3$, $\beta=1$ and $E(0)=0.1$.}
		\label{fig:mass_gamma_sigma}
	\end{subfigure}
	
	\vspace{0.3cm}
	
	\begin{subfigure}[b]{0.48\textwidth}
		\centering
		\includegraphics[width=\textwidth]{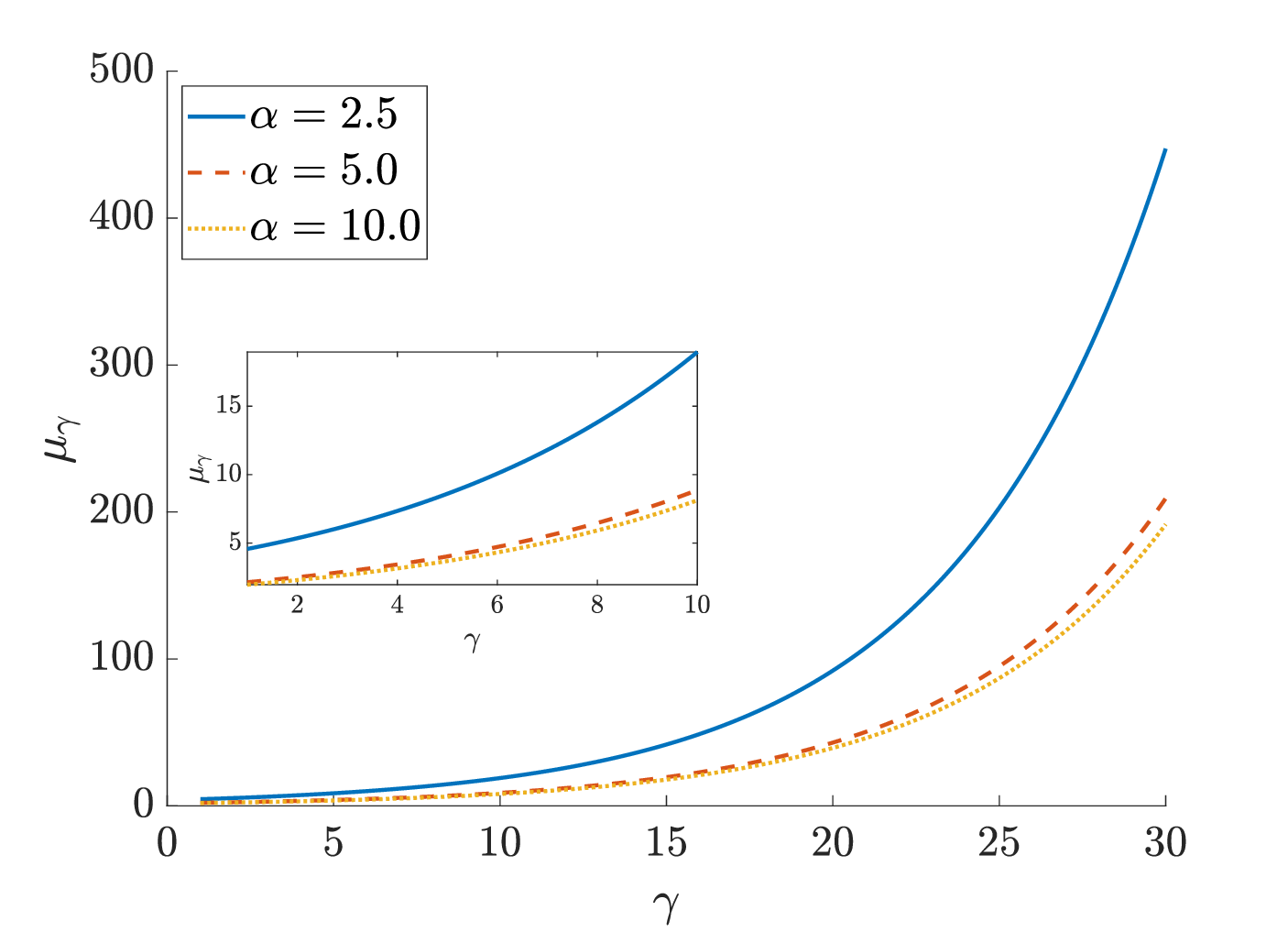}
		\caption{Minimal initial mass as a function of $\gamma$ for different values of $\alpha$ and with fixed $\sigma^2=0.025$, $\beta=1$ and $E(0)=0.1$.}
		\label{fig:mass_gamma_alpha}
	\end{subfigure}
	\hfill
	\begin{subfigure}[b]{0.48\textwidth}
		\centering
		\includegraphics[width=\textwidth]{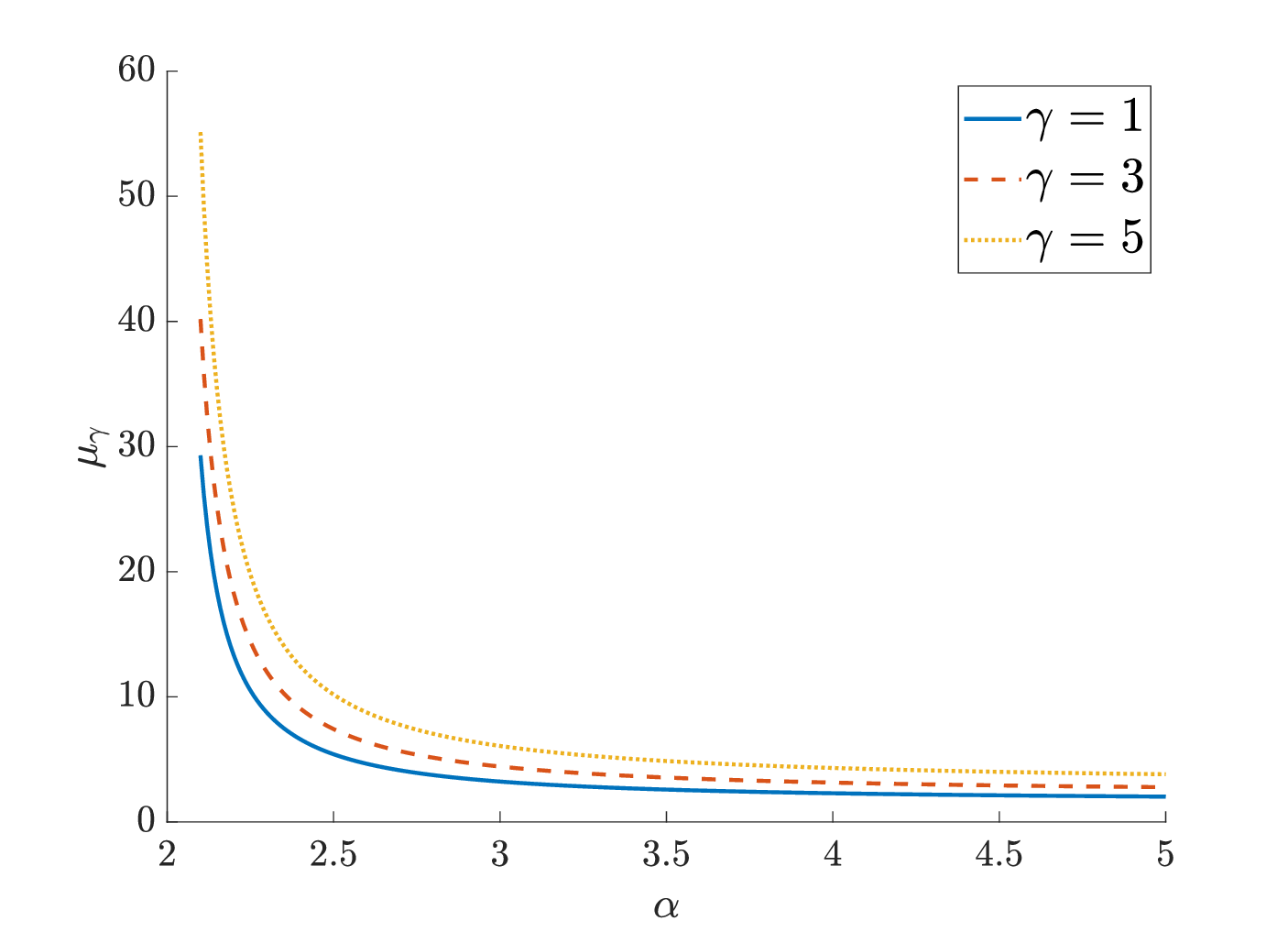}
		\caption{Minimal initial mass as a function of $\alpha$ for different values of $\gamma$ and fixed $\sigma^2=0.025$, $\beta=1$ and $E(0)=0.1$.}
		\label{fig:mass_alpha}
	\end{subfigure}
	
	\caption{Critical initial mass required for finite-time condensation in superlinear Fokker--Planck dynamics, showing its dependence on initial energy $E(0)$, nonlinearity exponent $\alpha$, diffusion shape parameter $\gamma$ and diffusion coefficient $\sigma^2$.}
	\label{fig:critical_mass_all}
\end{figure}

\section*{Conclusion}
In this work we investigated superlinear Fokker--Planck equations arising in consensus formation models with condensation effects, focusing on the role of the total mass for a wide class of diffusion coefficients. Inspired by kinetic descriptions of Bose--Einstein statistics, we considered nonlinear drift terms leading to aggregation and diffusion coefficients vanishing at the boundary of the opinion domain. We first analysed the structure of stationary solutions for a class of polynomial diffusion weights. We showed that, similarly to the classical case, stationary states may exhibit singular behaviour when a critical value of the normalisation constant is considered. In this regime, a finite critical mass exists whenever the superlinearity exponent satisfies $\alpha>2$, and the steady state develops a condensation around the mean opinion. We further demonstrated that this phenomenon persists for increasingly localised diffusion profiles, although the critical mass required to trigger condensation grows with the localisation parameter.

We then addressed the dynamical counterpart of this behaviour by studying the propagation of $L^2$-regularity for the time-dependent problem. By extending energy-based techniques previously developed for the standard diffusion case, we proved that solutions cannot remain uniformly bounded in $L^2$ whenever the initial mass exceeds a suitable threshold. This result establishes a finite-time loss of regularity driven by supercritical mass and provides a link between kinetic mass concentration and the formation of singular states. Our analysis highlights how nonconstant diffusion weights mitigate aggregation effects by increasing the mass threshold for condensation. Several questions remain open, including the investigation of trends to equilibrium and possible multidimensional extensions.

\section*{Acknowledgements}
The results obtain in the present paper have been partially presented in WASCOM 2025, University of Parma, Italy. The research underlying this paper has been undertaken within the activities of the GNFM group of INdAM (National Institute of High Mathematics). M.Z. acknowledges partial support from the PRIN2022PNRR project No.P2022Z7ZAJ, European Union-NextGenerationEU. M.Z. acknowledges partial support by ICSC - Centro Nazionale di Ricerca in High Performance Computing, Big Data and Quantum Computing, funded by European Union-NextGenerationEU. M.C. acknowledges the support by Fondo Italiano per la Scienza (FIS2023-01334) advanced grant "ADvanced numerical Approaches for MUltiscale Systems with uncertainties" - ADAMUS.

\section*{Declarations}
The authors declare no competing interests.

\end{document}